\documentclass[reqno]{amsart}
\usepackage{hyperref,amssymb}

\begin{document}
\title[Nonlocal Robin Laplacian]
{Generalized Nonlocal Robin Laplacian on Arbitrary Domains}
\author[N. Ait OUssaid]
{Nouhayla Ait Oussaid}

\author[K. Akhlil]
{Khalid Akhlil}

\author[S. Ben Aadi]
{Sultana Ben Aadi}

\author[M. El Ouali]
{Mourad El Ouali}

\address{\newline Nouhayla Ait Oussaid, Khalid Akhlil, Sultana Ben Aadi, and Mourad El Ouali\newline
Departement of Mathematics and Management\newline
POlydisciplinary Faculty of Ouarzazate\newline
Ibno Zohr University\newline
Agadir, Morocco.}
\email{~\newline nouhayla.aitoussaid@gmail.com\newline
k.akhlil@uiz.ac.ma\newline 
 sultana.benaadi@edu.uiz.ac.ma\newline
 mauros1608@;gmail.com}



\begin{abstract}
In this paper, we prove that it is always possible to define a realization of the Laplacian $\Delta_{\kappa,\theta}$ on $L^2(\Omega)$  subject to nonlocal Robin boundary conditions with general jump measures on arbitrary open subsets of $\mathbb R^N$. This is made possible by using a capacity approach to define admissible pair of measures $(\kappa,\theta)$ that allows  the associated form $\mathcal E_{\kappa,\theta}$ to be closable. The nonlocal Robin Laplacian $\Delta_{\kappa,\theta}$ generates a sub-Markovian $C_0-$semigroup on $L^2(\Omega)$ which is not dominated by Neumann Laplacian semigroup unless the jump measure $\theta$ vanishes. Finally, the convergence of semigroups sequences $e^{-t\Delta_{\kappa_n,\theta_n}}$ is investigated in the case of vague convergence and $\gamma-$convergence of admissible pair of measures $(\kappa_n,\theta_n)$.
\end{abstract}

\maketitle \setlength{\textheight}{19.5 cm}
\setlength{\textwidth}{12.5 cm}
\newtheorem{theorem}{Theorem}[section]
\newtheorem{lemma}[theorem]{Lemma}
\newtheorem{proposition}[theorem]{Proposition}
\newtheorem{corollary}[theorem]{Corollary}
\theoremstyle{definition}
\newtheorem{definition}[theorem]{Definition}
\newtheorem{example}[theorem]{Example}
\theoremstyle{remark}
\newtheorem{remark}[theorem]{Remark}
\numberwithin{equation}{section} \setcounter{page}{1}

\newcommand{\hth}{\hat{\theta}}
\newcommand{\Hun}{H^1(\Omega)}
\newcommand{\RR}{\mathbb R}
\newcommand{\ombar}{\overline{\Omega}}
\newcommand{\po}{\partial\Omega}
\newcommand{\me}{a}
\newcommand{\mf}{\mathcal F}
\newcommand{\mfm}{\mathcal F^{\kappa}}
\newcommand{\mem}{a_{\kappa}}
\newcommand{\Huntild}{\widetilde{H}^1(\Omega)}
\newcommand{\pmu}{\mathcal P_t^{\kappa}}
\newcommand{\ma}{\mathcal A}
\newcommand{\emu}{e^{-t\Delta_{\kappa}}}
\newcommand{\eN}{e^{-t\Delta_{N}}}
\newcommand{\eD}{e^{-t\Delta_{D}}}
\newcommand{\rcap}{\mathrm{Cap}_{\ombar}}
\newcommand{\dmu}{\Delta_{\kappa}}
\newcommand{\emup}{e^{-t\Delta_{\kappa^+}}}
\newcommand{\emun}{e^{-t\Delta_{-\kappa^-}}}
 \newcommand{\memp}{a_{\kappa^+}}
 \newcommand{\memn}{a_{-\kappa^-}}
\newcommand{\mfmp}{\mathcal F^{\kappa^+}}
\newcommand{\Capp}{\mathrm{Cap}}


\section{Introduction}

For sufficiently smooth open set $\Omega$, it is always possible to define the realization in $L^2(\Omega)$ of Laplace operator $-\Delta$ with Dirichlet boundary conditions, $u=0$ on $\po$, Neumann boundary conditions, $\partial_\nu u=0$ on $\po$, and Robin boundary conditions, $\partial_\nu u+\beta u=0$ on $\po$. These realizations are the self-adjoint operators on  $L^2(\Omega)$ associated with the closed bilinear forms
\begin{align*}
\mathcal E_D(u,v)=&\int_\Omega\nabla u\,\nabla v\,dx, \qquad u,\,v\in D(\mathcal E_D)=H^1_0(\Omega),\\
\mathcal E_N(u,v)=&\int_\Omega\nabla u\,\nabla v\,dx, \qquad u,\,v\in D(\mathcal E_N)=H^1(\Omega),
and \\
\mathcal E_R(u,v)=&\int_\Omega\nabla u\,\nabla v\,dx+\int_{\po}\beta uv\,d\sigma, \quad u,\,v\in D(\mathcal E_R)=H^1\Omega)
\end{align*} respectively. Here $\beta\in L^\infty(\po)$ is a nonnegative given function and $\sigma$ the surface measure. Each of the self-adjoint operators $\Delta_D$, $\Delta_N$ and $\Delta_R$, associated with the three closed forms respectively, generates a sub-Markovian $C_0-$semigroup on $L^2(\Omega)$. It is worth to mention that Robin boundary conditions generalize Dirichlet and Neumann ones by simply taking $\beta$ as infinity and zero respectively. Even more, Robin boundary conditions can be seen as boundary perturbation of Neumann boundary conditions. Moreover, the domination
\begin{equation}\label{dom1}
0\leq e^{-t\Delta_D}\leq e^{-t\Delta_R}\leq e^{-t\Delta_N},\qquad (t\geq 0)\end{equation} holds in the positive operators sense.

If $\Omega$ is an arbitrary bounded open set, Daners \cite{D00} replaced the surface measure $\sigma$ by the $(N-1)-$dimensional Hausdorff measure $\mathcal H^{N-1}$ and showed a way how to define a self-adjoint realization of the Laplacian which satisfies in a weak sense Robin boundary condition. Arendt and Warma \cite{AW1} developed an alternative approach to define self-adjoint realization of the Laplacian with even more general Robin boundary conditions; i.e., boundary conditions defined by arbitrary Borel measures. For this aim, the authors used systematically the notion of relative capacity; i.e., the capacity induced by the form $\mathcal E(u,v)=\int_\Omega\nabla u\,\nabla v\,dx$ with domain $\mathcal D:=\Huntild$: the completion of $H^1(\Omega)\cap C_c(\ombar)$ in $H^1(\Omega)$. The choice of the domain $\mathcal D$ of $\mathcal E$ is justified by the fact that  the form inducing a capacity needs to be regular; i.e., $D(\mathcal E)\cap C_c(\ombar)$ is dense in $D(\mathcal E)$ and in $C_c(\ombar)$. The second density is true due to the Stone-Weierstrass theorem, but if $D(\mathcal E)$ was equal to $H^1(\Omega)$ then the first density needs not to be true. This is due to a lack of regularity of the domain $\Omega$. So the definition of $\mathcal D$ is a way to "regularize" the form $(\mathcal E,H^1(\Omega))$ and then to be able to take advantage of all the potentialities of the underlying (relative) capacity. 

Now, given a measure $\kappa$, the authors in \cite{AW1} (inspired by \cite{AS, S, SV}) proved the existence of a realization of Laplace operator with general Robin boundary conditions by considering on $L^2(\Omega)$ the positive form $\mathcal E_\kappa$ defined by
\[
\mathcal E_\kappa(u,v)=\int_\Omega\nabla u\nabla v\,dx+\int_{\po} u v\,d\kappa
\]with domain $D(\mathcal E_\kappa)=H^1(\Omega)\cap C_c(\ombar)\cap L^2(\po,d\kappa)$. The closability of the form $(\mathcal E_\kappa, D(\mathcal E_\kappa))$ is acquired if and only if the measure $\kappa$ is admissible; i.e., $\kappa$ charges no set with relative capacity zero (relative polar sets). Moreover, the self-adjoint operator $\Delta_\kappa$ associated with the closure of $(\mathcal E_\kappa, D(\mathcal E_\kappa))$ generates a sub-Markovian $C_0-$semigroup on $L^2(\Omega)$ which is, similarly to \eqref{dom1}, sandwiched between Dirichlet semigroup and Neumann semigroup.

The main goal of the present paper is to prove, by means of the aforementioned capacity approach, the existence of the realization in $L^2(\Omega)$ of the Laplacian with nonlocal Robin boundary conditions involving general jump measures on the boundary.  Given a Borel measure $\kappa$ on $\partial\Omega$ and a symmetric Radon measure $\theta$ on $\partial\Omega\times\partial\Omega\setminus_d$ where $d$ indicates the diagonal of $\partial\Omega\times\partial\Omega$, we define the bilinear symmetric form $\mathcal E_{\kappa,\theta}$ on $L^2(\Omega)$ by

\[
 \mathcal E_{\kappa,\theta}(u,v)=\int_{\Omega}\nabla u\nabla v dx+\int_{\po}uvd\kappa+\int_{\partial\Omega\times\partial\Omega\setminus_d}(u(x)-u(y))(v(x)-v(y))d\theta,
\]whenever the right hand side is meaningful in $\Hun\cap C_c(\ombar)$. The first question to answer in this paper is whether the symmetric positive form $a_{\kappa,\theta}$ is closable. This will be done by introducing the notion of admissible pair of measures; i.e. the measure $\kappa+\hat\theta$ charges no relative polar sets, where $\hat\theta(dx):=\theta(dx,\po)$. This shows that the closablility of the form $a_{\kappa,\theta}$ is inherent to the measure $\hat\theta$ rather than to the measure $\theta$ itself. The associated self-adjoint operator $\Delta_{\kappa,\theta}$ generates a sub-Markovian semigroup which is dominated by Dirichlet Laplacian semigroup but not by Neumann Laplacian semigroup; i.e.
\begin{equation}\label{dom2}
0\leq e^{-t\Delta_D}\leq e^{-t\Delta_{\kappa,\theta}}\nleq e^{-t\Delta_N},\qquad (t\geq 0)\end{equation}
holds in the positive operators sense, which means that the right hand side in the sandwiched property \eqref{dom1} fails. This is due essentially to the nonlocality term expressed with the jump measure $\theta$.

Many authors have investigated local and nonlocal Robin boundary conditions. For local boundary conditions, see e.g. \cite{A15, A12, A18, AW1, AW2, CW12, D00, W02}, where the existence of the realization of Robin Laplacian, regularity, domination results and qualitative properties of the semigroup are the main discussed subjets. The linear and nonlinear nonlocal Robin boundary conditions have captured a particular attention in recent years, see \cite{GM09, SW10, S11, S12}. All previous works about linear nonlocal Robin boundary conditions deal with jump measures of the form
\[\theta(dx,dy)=\frac{1}{|x-y|^{N+2s}}\kappa(dx)\kappa(dy),\quad x,\,y\in\po,\]
while here we consider general symmetric Radon measures $\theta$ which justifies the appellation "general" nonlocal Robin boundary conditions.

The rest of this paper is structured as follows. In Section 2, we recall the concept of relative capacity and its properties and some  results from forms theory. In Section 3, we give a characterization  of the realization on $L^2(\Omega)$ of the Laplacian with general nonlocal Robin boundary conditions. In Section 4, we prove that the self-adjoint operator $\Delta_{\kappa,\theta}$ generates a sub-Markovian $C_0-$semigroup on $L^2(\Omega)$. The results about the domination of the semigroup $e^{-\Delta_{\kappa,\theta}}$ is detailed in Section 5. We conclude the paper by Section 6, where we give some convergence results of semigroups sequences with respect to the admissible pair of measures.

\section{Preliminaries}
\subsection{Relative Capacity}
Let $\Omega$ be an open bounded set of $\mathbb R^N$. The following space is fundamental in the rest of this paper
\[
\Huntild=\overline{\Hun\cap C_c(\ombar)}^{\Hun}
\]where $C_c(\ombar)$ denotes the space of all continuous real-valued functions with compact support in $\ombar$ and $\Hun$ the usual Sobolev space endowed with its usual norm. For conditions implying that $\Hun=\Huntild$ see \cite[Section 1.1.6, Theorem 2]{M85} or \cite[Chap V, Theorem 4.7]{EE90}. In particular, if $\Omega$ has a lipschitz boundary, then $\Hun=\Huntild$. It is worth to mention that the equality does not hold generally, take for example $\Omega=(0,1)\cup(1,2)$

Given an arbitrary subset of $\mathbb R^N$, the \textbf{capacity} $\Capp(A)$ of $A$ is defined as
\begin{align*}
\mathrm{Cap}(A)=\inf\{\|u\|_{H^1(\mathbb R^N)}\,:\, u\in H^1(\mathbb R^N),\,\exists\, O\subset\mathbb R^N\,\text{ open such that }\\ A\subset O \text{ and } u(x)\geq 1 \text{ a.e. on }O\}.
\end{align*}

The \textbf{relative capacity} $\rcap(A)$ (i.e., with respect to the fixed open set $\Omega$) is defined for an arbitrary open set $A$ of $\ombar$ by 

\begin{align*}
\rcap(A)=\inf\{\|u\|_{H^1(\Omega)}\,:\, u\in\Huntild),\,\exists\, O\subset\mathbb R^N\,\text{ open such that }\\ A\subset O \text{ and } u(x)\geq 1 \text{ a.e. on }\Omega\cap O\}.
\end{align*}

We consider the topological space $X=\ombar$, the $\sigma-$algebra $\mathcal B(X)$ of all Borel  sets in $X$, and the measure $m$ on $\mathcal B(X)$ given by $m(A)=\lambda(A\cap \Omega)$ for all $A\in\mathcal B(X)$ with $\lambda$ the Lebesgue measure. Denoting by $L^2(\Omega)$ the usual $L^2-$space with respect to the Lebesgue measure, we then have
\[
L^2(\Omega)=L^2(X,\mathcal B(X),m).
\]
The introduction of $m$ is needed to ensure this identity in the case where $\partial \Omega$ has positive Lebesgue measure.
%
Moreover, we have that $\Huntild\cap C_c(\ombar)$ is dense in $\Huntild$. Thus the Dirichlet form $(\mathcal E,\mathcal D)$ on $L^2(X,\mathcal B(X),m)$ is regular. As introduced in \cite{AW1} the relative capacity of a subset $A$ of $\ombar$ is exactly the capacity of $A$ associated with the Dirichlet form $(\mathcal E,\mathcal D)$ on $L^2(X,\mathcal B(X),m)$ in the sense of \cite[8.1.1, p.52]{BH91}. Note that open subsets of $X=\ombar$ are understood with respect to the topology of $X$. Thus $\rcap$ (just as $\Capp$) has the property of a capacity as described in \cite[I.8]{BH91} or \cite[Chapter 2]{FOT94}.
\begin{definition}
\begin{enumerate}
\item[(1)] A subset $A$ of $\mathbb R^N$ is called polar set if $\Capp(A)=0$. A subset $A$ of $\ombar$ is called relatively polar set if $\rcap(A)=0$.

\item[(2)] We say that a property holds on $\ombar$ relatively quasi-everywhere (r.q.e.), if it holds for all $x\in\ombar\setminus N$ where $N\subset\ombar$ is relatively polar.

\item[(3)] A scalar function $u$ defined on $\mathbb R^N$ is called quasi-continuous if for each $\varepsilon>0$ there exists an open set $G\subset\mathbb R^N$ such that $\rcap(G)<\varepsilon$ and $u$ is continuous on $\mathbb R^N\setminus G$.

\item[(4)] A scalar function $u$ on $\ombar$ is called relatively quasi-continuous, if for each $\varepsilon>0$ there exists an open set $G\subset\mathbb R^N$ such that $\rcap(G\cap \ombar)<\varepsilon$ and $u$ is continuous on $\ombar\setminus G$.

\end{enumerate}

\end{definition}

\begin{remark}
\begin{enumerate}
\item We know that for sets in $\Omega$, the notion of polar sets and relatively polar sets are equivalent. Thus, merely subsets of $\po$ are of interest for our investigations.

\item As $\rcap(A)\leq\Capp(A)$ for any $A\subset\ombar$, it follows that $u_{|_{\ombar}}$ is relatively quasi-continuous. Now, it follows from \cite[Proposition 8.2.1]{BH91} that for each $u\in\Huntild$ there exists a relatively quasi-continuous function $\tilde u:\ombar\rightarrow\mathbb R$ such that $\tilde u(x)=u(x)$ $m-$a.e. We call $\tilde u$ the relatively quasi-continuous representative of $u$. It can be seen from the proof of \cite[I, Proposition 8.2.1]{BH91} that $\tilde u$ can be chosen Borel measurable.

\end{enumerate}

\end{remark}

Let $\Omega\subset\mathbb R^N$ be bounded open set and define
\[
H^1_0(\Omega):=\overline{\mathcal D(\Omega)}^{H^1(\Omega)}
\]where $\mathcal D(\Omega)$ denotes the space of all infinitely differentiable functions with compact support. We end this subsection with a characterization of $H^1_0(\Omega)$ with the help of relative capacity\cite{AW1}, that is
\[
H^1_0(\Omega)=\{u\in\Huntild\,:\, \tilde u(x)=0\, \text{ r.q.e on }\po\}.
\]

\subsection{Dirichlet forms}

Let $H$ be a real Hilbert space. A positive form on $H$ is a bilinear mapping $\mathcal E: D(\mathcal E)\times D(\mathcal E)\rightarrow\mathbb R$ such that $\mathcal E(u,v)=\mathcal E(v,u)$ and $\mathcal E(u,u)\geq 0$ for all $u,\,v\in D(\mathcal E)$ (the domain of the form is a dense subspace of $H$). The form is closed if $D(\mathcal E)$ is complete for the norm $\|u\|_\mathcal E=\left(\mathcal E(u,u)+\|u\|_H^2\right)^{1/2}$. Then the operator $A$ on $H$ associated with $\mathcal E$ is defined by
\begin{equation*}
  \left\{
  \begin{aligned}
D(A):=&\{u\in D(\mathcal E)\,:\,\exists v\in H\,\, \mathcal E(u,\varphi)=(v,\varphi)_H\,\,\forall \varphi \in D(\mathcal E)\},\\
Au:=& v
\end{aligned}
 \right.
\end{equation*}

The operator $A$ is self-adjoint and $-A$ generates a contraction semigroup $(e^{-tA})_{t\geq 0}$ of symmetric operators on $H$. Moreover, the form $\mathcal E$ is called closable if for each Cauchy sequence $(u_n)_{n\in\mathbb N}$ in $(D(\mathcal E),\| \,.\, \|_{\mathcal E})$, $\lim_{n\to\infty} u_n=0$ in $H$ implies $\lim_{n\to\infty} \mathcal E(u_n,u_n)=0$. In that case the closure $\bar{\mathcal E}$ is the unique positive closed form extending $\mathcal E$ such that $D(\mathcal E)$ is dense in $D(\bar{\mathcal E})$.

 Given a positive form $\mathcal E$ there always exists by \cite[Theorem S15, p.373]{RS80} a closable positive form $\mathcal E_r\leq \mathcal E$ such that $\mathcal F\leq \mathcal E_r$ whenever $\mathcal E$ is a closable positive form such that $\mathcal F\leq \mathcal E$. Thus $\mathcal E_r$ is the largest closable form smaller or equal to $\mathcal E$. Clearly $\mathcal E$ is closable if and only if $\mathcal E=\mathcal E_r$.


Now assume that $H=L^2(\Omega)$ where $(\Omega,\Sigma,\lambda)$ is a $\sigma-$finite measure space. We let $L^2(\Omega)_+=\{f\in L^2(\Omega)\,:\,\,\,\, f\geq 0\,\text{ a.e.}\}$ and $F_+=L^2(\Omega)_+\cap F$ if $F$ is a subspace of $L^2(\Omega)$.

\begin{theorem}[First Beurling-Deny criterion]
Let $S$ be the semigroup associated with a closed, positive form $\mathcal E$ on $L^2(\Omega)$. Then $S$ is positive (i.e. $S(t) L^2(\Omega)_+\subset L^2(\Omega)_+$ for all $t\geq 0$) if and only if $u\in D(\mathcal E)$ implies $|u|\in D(\mathcal E)$ and $\mathcal E(|u|)\leq \mathcal E(u)$. 
\end{theorem}

\begin{theorem}[Second Beurling-Deny criterion]
Let $S$ be the semigroup associated with a closed, positive form $\mathcal E$ on $L^2(\Omega)$. Assume that $S$ is positive. Then the   asserts that $S$ is $L^\infty-$contractive (i.e., if $f\in L^2(\Omega)$ satisfy $0\leq f\leq 1$ then $0\leq S(t)f\leq 1$ for all $t\geq 0$) if and only $0\leq u\in D(\mathcal E)$ implies $u\wedge 1\in D(\mathcal E)$ and $a(u\wedge 1)\leq \mathcal E(u)$.

\end{theorem}

We say that the semigroup $S$ is \textbf{sub-Markovian} if it is positive and $L^\infty-$contractive.  A Dirichlet form is a closed positive form satisfyin the two Beurling-Deny criterion.

\begin{remark}
Notice that several authors \cite{D89,BH91,FOT94} call a symmetric Markov semigroup what authors in \cite{AW1, AW2} call a symmetric submarkovian semigroup on $L^2(\Omega)$. In the current paper, we choose the later appellation.
\end{remark}
Now let $\mathcal E$ and $\mathcal F$ be two closed, positive forms on $L^2(\Omega)$ such that the associated semigroup $S$ and $T$ are positive. We say that $D(\mathcal E)$ is an ideal of $D(\mathcal B)$ if
\begin{itemize}
\item[a) ] $u\in D(\mathcal E)$ implies $|u|\in D(\mathcal F)$ and,
\item[b) ] $0\leq u\leq v$, $v\in D(\mathcal E)$, $u\in D(\mathcal F)$ implies $u\in D(\mathcal E)$.
\end{itemize}

\textbf{Ouhabaz's domination criterion} \cite{O96} says that
\[
0\leq S(t)\leq T(t)\quad (t\geq 0)
\]if and only if $D(\mathcal E)$ is an ideal of $D(\mathcal F)$ and 
\[
\mathcal E(u,v)\geq \mathcal F(u,v)\quad \text{for all }u,\,v\in D(\mathcal E)_+.
\]

\section{Closability}

Let $\Omega$ be an open bounded set of $\mathbb R^N$. Let $\kappa$  be a Borel measure on $\partial\Omega$ and $\theta$ a symmetric Radon measure on $\partial\Omega\times\partial\Omega\setminus_d$ where $d$ indicates the diagonal of
$\partial\Omega\times\partial\Omega$. Let
\[
 E=\{u\in H^1(\Omega)\cap C_c(\overline{\Omega}):\int_{\po}|u|^2d\kappa+\int_{\partial\Omega\times\partial\Omega\setminus_d}(u(x)-u(y))^2d\theta<\infty\}
\]

We define the bilinear symmetric form $\mathcal E_{\kappa,\theta}$ with domain $E$ on
$L^2(\Omega)$ by

\[
 \mathcal E_{\kappa,\theta}(u,v)=\int_{\Omega}\nabla u\nabla v dx+\int_{\po}uvd\kappa+\int_{\partial\Omega\times\partial\Omega\setminus_d}(u(x)-u(y))(v(x)-v(y))d\theta
\]

It is natural to ask whether $(a_{\kappa,\theta},E)$ is closable in $L^2(\Omega)$ or not? Remark that if one tries to prove its closability, it will be needed to prove that $u_{n{|_{\po}}}$ converges to zero in $L^2(\Omega)$ and $[u_{n{|_{\po}}}]$ converges also to zero in $L^2(\po\times\po\setminus d)$($[u]$ is defined by $[u](x,y)=u(x)-u(y)$). The following examples show that the aforementioned convergences don't always hold.

\begin{example}Let $\Omega$ be a bounded domain in $\mathbb R^N$ with Lipschitz boundary. Fix $z,\,z'\in\po$ and let $d\kappa=d\delta_z$ and $d\theta=\frac{1}{2}\left(d\delta_z\otimes \delta_{z'}+d\delta_{z'}\otimes \delta_z\right)$, where $\delta_a$ is the Dirac measure in $a$ and $d\delta_z\otimes \delta_{z'}(x,y)=d\delta_z(x)\, d\delta_{z'}(y)$. For such measures we have
\[
\mathcal E_{\kappa,\theta}(u)=\int_\Omega |\nabla u|^2dx+u^2(z)+(u(z)-u(z'))^2
\]
Since $\Hun\cap C_c(\ombar)$ is dense in $\Hun$, there exists $u_n\in \Hun\cap C_c(\ombar)$ such that $u_n\rightarrow 0$ in $\Hun$ as $n\to\infty$ with $u_n(z)=1$ and $u_n(z')=0$ for all $n\geq 1$. For such sequence, we have $\lim_{n\,m\to\infty} a_{\kappa,\theta}(u_n-u_m)=0$ but $\lim_{n\,m\to\infty} \mathcal E_{\kappa,\theta}(u_n)=2$.
\end{example}

\begin{example}
If the jump measure $\theta=0$, then the form $\mathcal E_{\kappa,0}$ is not closable in general. In \cite{AW1}, it is proved that $\mathcal E_{\kappa,0}$ is closable if the measure $\kappa$ is admissible; i.e., for each Borel set $A$, $\rcap(A)=0$ $\Rightarrow$ $\kappa(A)=0$.
\end{example}

\begin{example}
If the jump measure is $\theta(dx,dy)=\frac{1}{|x-y|^{N+2s}}\kappa(dx)\kappa(dy)$, then the form $\mathcal E_{\kappa,\theta}$ is closable if and only if the measure $\kappa$ is admissible, see \cite{CW20}.
\end{example}

The examples show that the form $(\mathcal E_{\kappa,\theta},E)$ is not closable in general, but as it is a positive form, then with means of a Reed-Simon Theorem \cite[Theorem S15, p:373]{RS80},
we know that there exist a closable positive form $(\mathcal E_{\kappa,\theta})_r\leq\mathcal E_{\kappa,\theta}$  such that $\mathcal F\leq(\mathcal E_{\kappa,\theta})_r $ whenever $\mathcal F$ is a closable form such that $\mathcal F\leq \mathcal E_{\kappa,\theta}$. Thus $(\mathcal E_{\kappa,\theta})_r $ is the largest closable form smaller or equal than $\mathcal E_{\kappa,\theta}$. Clearly, $\mathcal E_{\kappa,\theta}$ is closable if and only if $\mathcal E_{\kappa,\theta} =(\mathcal E_{\kappa,\theta})_r $.

Define the Borel measure $\hat{\theta}(dx):=\theta(dx\times\po\setminus d)$. We want to determine $(\mathcal E_{\kappa,\theta})_r$ and in particular characterize when $\mathcal E_{\kappa,\theta}$ is closable. One can see that the closability of the form $\mathcal E_{\kappa,\theta}$ is inherent to the measure $\hat{\theta}$ rather than to the measure $\theta$. We say that a Borel measure $\nu$ on $\po$ is locally infinite everywhere if 
\[
\forall z\in\po\text{ and } r>0\quad \nu(\po\cap B(z,r))=\infty
\]

\begin{proposition}
Let $\kappa$ and $\theta$ be Borel measures on $\po$ and $\po\times\po\setminus d$ respectively. We assume that $\kappa$ or $\hat \theta$ are locally infinite everywhere on $\po$. Then the form $\mathcal E_{\kappa,\theta}$ is closable and its closure which we denote by $\mathcal E_\infty$ is given by
\[
\mathcal E_\infty(u,v)=\int_\Omega\nabla u\,\nabla v\,dx
\]with domain $H^1_0(\Omega)$.

\end{proposition}
\begin{proof}
Let $u\in E$. We have
\begin{align*}
&\int_{\po}|u|^2d\kappa+\int_{\partial\Omega\times\partial\Omega\setminus_d}(u(x)-u(y))^2d\theta\\
&=\int_{\po}u^2(x) \kappa(dx)+2\int_{\partial\Omega\times\partial\Omega\setminus_d}u^2(x) \theta(dx,dy)-2\int_{\partial\Omega\times\partial\Omega\setminus_d} u(x)u(y) \theta(dx,dy)\\
&=\int_{\po}u^2(x)\kappa(dx)+2\int_{\po} u^2(x)\int_{\po} \theta(dx,dy)-2\int_{\partial\Omega\times\partial\Omega\setminus_d} u(x)u(y) \theta(dx,dy)\\
&=\int_{\po}u^2(x) \kappa(dx)+2\int_{\po}u^2(x)\hat\theta(dx)-2\int_{\partial\Omega\times\partial\Omega\setminus_d} u(x)u(y) \theta(dx,dy)
\end{align*}
Since $u$ is continuous on $\ombar$, if $\kappa$ or $\hat \theta$ are locally infinite everywhere on $\po$, it then follows that $u_{|_{\po}}=0$ and thus $E=\{u\in H^1(\Omega)\cap C_c(\ombar)\,:\, u_{|_{\po}}=0\}$. One obtains that for all $u,\,v\in E$,
\[
\mathcal E_{\infty}(u,v):=a_{\kappa,\theta}(u,v)=\int_\Omega\nabla u\,\nabla v\,dx
\]It is clear that $(\mathcal E_\infty,E)$ is closable in $L^2(\Omega)$. It is still to be proved that $E\subset H^1_0(\Omega)$. To do so it suffices to follows line by line the second part of the proof of \cite[Proposition 3.2.1]{W02}.

\end{proof}

Let
\[
 \Gamma_{\kappa,\theta}:=\{z\in\po:\exists r>0\text{ such that } \kappa(\po\cap B(z,r))+\hat{\theta}(\po\cap B(z,r))<\infty\}
\]
be the part of $\po$ on which $\kappa$ and $\hat{\theta}$ are locally finite. It is easy to verify that  $\tilde u=0$ on $\po\setminus \Gamma_{\kappa,\theta}$ for all $u\in E$. Since $\Gamma_{\kappa,\theta}$ is a locally compact metric space, it follows from \cite[Theorem 2.18 p.48]{R66} that $\kappa$ and $\hat\theta$ are regular Borel measures on $\Gamma_{\kappa,\theta}$. Therefore $\kappa$ and $\hat \theta$ are Radon measures on $\Gamma_{\kappa,\theta}$. We introduce the notion of admissible pair of measures. We say that \textit{the pair $(\kappa,\theta)$ is admissible} if for each Borel set $A\subset\Gamma_{\kappa,\theta}$ one has
\[
 \rcap(A)=0\Rightarrow (\kappa+\hat{\theta})(A)=0
\]where $\rcap(A)$ is the relative capacity of the subset $A$.

The following Lemma is an adapation of \cite[Lemma 4.1.1]{Br}, and
will be used to prove Theorem \ref{th:1}.

\begin{lemma}\label{lem:1}
 Let $B$ be a Borel set of $\po$ such that $\hat{\theta}(B)>0$. Then there exists disjoint compact sets $K$ and $C$ such that $K\subset B$ and $\theta(K\times C)>0$
\end{lemma}
\begin{proof}
By definition of positive Radon measures, $\theta$ is inner regular. Let $B$ be a Borel set of $\po$ such that $\hat\theta(B)=\theta(\po\times B)>0$. It is then possible to choose a set $S\subset\po\times B\setminus d$ which is compact as a subset of $\po\times\po\setminus d$ and therefore compact as a subset of $\po\times\po$ such that $\theta(S)>0$. 

We choose a metric $\rho_0$ which is compatible with the topology of $\po$. Then the metric $\rho$ on $\po\times\po$, given by
\[
 \rho( (x,y);(x',y') ):=\rho_0(x,x')+\rho_0(y,y'),\quad \forall x,x',y,y'\in\po
\]is compatible with the product topology on $\po\times\po$.

Since $S$ is compact and the diagonal $d$ is closed we can choose a real number $a$ such that $0<2a<dist(S,d)$.

Since $\po$ is relativly compact each $x\in\po$ has a compact neighborhood $K_x$ with diameter less that $a$. Since $S$ can be covered by finitly many sets of the form $K_x\times K_y$ and $\theta(S)>0$, we can choose $(x_0,y_0)\in\po\times\po$ such that $\theta(S\cap K_{x_0}\times K_{y_0})>0$.

We claim that $K_{x_0}\cap K_{y_0}=\emptyset$. Suppose that there exist $z\in K_{x_0}\cap K_{y_0}$. Choose $(x',y')\in S\cap K_{x_0}\times K_{y_0}$. Then 
\begin{align*}
\rho((z,z);(x',y'))&= \rho_0(z,x')+\rho_0(z,y')\\
&\leq \quad a\quad+\quad a\\
&\leq 2a 
\end{align*}
This contradicts the fact that $dist(d,K_0)>2a$. 

Now set $C:=K_{x_0}$ and $K:=K_{y_0}\cap\{y\in\po:\exists x\in\po\text{ with }(x,y)\in S\}$. Then $C$ and $K$ are disjoint compact sets and $K \subset B$. Moreover $\theta(K\times C)>0$ since $K\times C\supset K_{x_0}\times K_{y_0}\cap K_0$.

\end{proof}

\begin{theorem}\label{th:1}
 The form $\mathcal E_{\kappa,\theta}$ is closable if and only if the pair $(\kappa,\theta)$ is admissible.
\end{theorem}

\begin{proof}
$(\Leftarrow)$ 
Let $u_k\in E$ be such that
$u_k\rightarrow 0$ in $L^2(\Omega)$ and $\lim\limits_{n,k\to\infty}\mathcal E_{\kappa,\theta}(u_n-u_k,u_n-u_k)=0$. It is clear that $u_k\rightarrow 0$ in $H^1(\Omega)$. By \cite[Theorem 2.1.3]{W02} applied to the relative capacity, the sequence $(u_k)$ contains a subsequence which converges to zero r.q.e. on $\ombar$. Since $\kappa+\hth$ charges no set of zero relative capacity, it follows that ${u_k}_{|\po}\rightarrow 0\text{ }\kappa$ a.e. and $\hth$ a.e. Since $u_k$ is a Cauchy sequence in $L^2(\po,\kappa)$, it follows that $u_k\rightarrow 0$ in $L^2(\po,\kappa)$.
 Since  ${u_k}_{|\po}\rightarrow 0\text{ }\hth$ a.e, thus $u_k(x)-u_k(y)\rightarrow 0\text{ }\theta$ a.e. Since $u_k(x)-u_k(y)$ is a Cauchy sequence in $L^2(\po\times\po\setminus d,\theta)$, it follows that  $u_k(x)-u_k(y)\rightarrow 0$ in $L^2(\po\times\po\setminus d,\theta)$ and thus $\lim\limits_{k\to\infty} \mathcal E_{\kappa,\theta}(u_k,u_k)=0$, which means that the form $(\mathcal E_{\kappa,\theta},E)$ is closable.

$(\Rightarrow)$ Suppose that there exists a Borel set $B$ such that
$\rcap(B)=0$ and $\hth(B)>0$ or $\kappa(B)>0$. We show that
$\mathcal E_{\kappa,\theta}$ is not closable. We consider just the case where $\hth(B)>0$.  The case where $\kappa(B)>0$ can be proved
simultaneously.

By Lemma \ref{lem:1} we can choose disjoint compact sets $K$ and $C$ such that $\theta(C\times K)>0$ and $K\subset B$(and therefore
$\rcap(K)=0$).

Since $\rcap(K)=0$, by \cite[Theorem 2.2.4]{W02}, there exists a
sequence $u_k\in H^1(\Omega)\cap C_c(\ombar)$ such that
\[
 0\leq u_k\leq 1,\quad u_k=1\text{ on }K\text{ and }\|u_k\|_{H^1(\Omega)}\to 0
\]

Let $(A_i)$ be a seqence of relatively open sets with compact
closure satisfaying
\[
 K\subset \overline{A_{i+1}}\subset A_i\subset\po\text{ and }\bigcap_i\overline{A_i}=K
\]

There exists then a sequence $v_i\in\mathcal D(\mathbb R^N)$ such
that $\mathrm{supp}[v_i]\subset A_i,\text{ }v_i=1\text{ on }K$ and
$0\leq v_i\leq 1$. Clearly, ${v_i}_{|\Omega}\in H^1(\Omega)\cap
C_c(\ombar)$ and $\|u_kv_i\|_{H^1(\Omega)}\to 0$ as $k\to\infty$.
For all $i\geq 1$ we have $u_kv_i\in H^1(\Omega)\cap C_c(\ombar)$.
Moreover, for all $i,k,$ $0\leq u_kv_i\leq 1$ and $u_kv_i=1$ on $K$.
For all $i\geq 1$, we choose $k_i\in \mathbb N$ such that
$\|u_{k_i}v_i\|_{H^1(\Omega)}\leq\frac{1}{2^i}$. Let
$w_i=u_{k_i}v_i$. Then $w_i\to 0$ in $H^1(\Omega)$, $0\leq w_i\leq
1$ and $w_i=1$ on $K$. Moreover $w_i\to\chi_K$ pointwise, since
$\mathrm{supp}[w_i]\subset A_i$.

One has  $\quad\sup_i\|w_i\|^2_{\kappa,\theta}<\infty$. In fact
\[\sup_i\int (w_i(x)-w_i(y))^2d\theta \leq 2\sup_i\int w_i^2(x)d\hat{\theta}\leq 2\hat{\theta}(K)<\infty
 \]
and\[ \sup_i\int w_i^2(x)d\kappa\leq\kappa(K)<\infty
\]
Since $K$ and $C$ are disjoint compact sets of $\po$ and
$H^1(\Omega)\cap C_c(\ombar)$ is dense in $
(C_c(\ombar),\|.\|_{\infty}) $ , we can choose $f\in H^1(\Omega)\cap
C_c(\ombar)$ such that $f(x)\geq 1$ for all $x\in K$ and
$|f(x)|\leq\frac{1}{4}$ for all $x\in C$. Since $\mathcal E_{\kappa,\theta}$ is sub-Markovian, its closure is a Dirichlet form. Thus, by \cite[Theorem
1.4.2 p:14]{BH91}
\[
 \|uv\|_{\kappa,\theta}\leq\|u\|_{\infty}\|v\|_{\kappa,\theta}+\|v\|_{\infty}\|u\|_{\kappa,\theta}\quad\forall u,v\in D(\mathcal E_{\kappa,\theta})
\]

We deduce that $\sup_i\|fw_i\|_{\kappa,\theta}<\infty$. Define
$h_n=\frac{1}{n}\sum_{i=1}^nfw_i$. Thus, selecting a subsequence if
necessary, we may assume that $(h_n)_{n\geq 0}$ is a Cauchy sequence
in $(\mathcal E_{\kappa,\theta},D(\mathcal E_{\kappa,\theta}))$, and therefore convergent in
$L^2(\Omega)$. Since $h_n\to 0$ a.e. we have $h_n\to 0$ in
$L^2(\Omega)$. Since $\mathcal E_{\kappa,\theta}$ is closable this implies that
$\mathcal E_{\kappa,\theta}(h_n,h_n)\to 0$.

In the other hand and  by the choice of $f$, we can choose $n_0$
large enough such that, for any $n\geq n_0$ we have $h_n(x)\geq
\frac{3}{4},\forall x\in K$ and $h_n(y)\leq\frac{1}{2},\forall y\in
C$. Thus
\[
 \limsup\limits_{n\to\infty}\iint(h_n(x)-h_n(y))^2d\theta\geq\frac{1}{16}\theta(K\times C)>0
\]
and therefore
$\limsup\limits_{n\to\infty}\mathcal E_{\kappa,\theta}(h_n,h_n)>0$. the
existence of such sequence contradicts the closability of
$\mathcal E_{\kappa,\theta}$.

\end{proof}

Let $\kappa$ and $\theta$ be two Borel measures on $\po$ and $\po\times\po\setminus d$ respectively such that the pair $(\kappa,\theta)$ is admissible. By definition, the domain of the closure of the form $(\mathcal E_{\kappa,\theta},E)$ is the completion of $E$ with respect to the $\mathcal E_{\kappa,\theta}-$norm. The following result gives its characterization. Before, note that, throughout the following, for $u\in \Huntild$, we will always choose the r.q.c. version $\tilde u$ which is Borel measurable. The proof of the following theorem is much the same as the similar result in \cite{W02} or in \cite{AW1}.

\begin{theorem}
 Let $\kappa$ (resp. $\theta$) be a Radon measure on $\po$ (resp. $\po\times\po\setminus_d$). Suppose that the pair $(\kappa,\theta)$ is admissible. Then the closure of $(\mathcal E_{\kappa,\theta},E)$ is given by
\begin{equation*}
 \mathfrak D=\{u\in \Huntild:\int_{\po}|\widetilde{u}|^2d\kappa+\frac{1}{2}\int_{\partial\Omega\times\partial\Omega\setminus_d}(\widetilde{u}(x)-\widetilde{u}(y))^2d\theta<\infty\}
\end{equation*}
\begin{equation*}
 \mathcal E_{\kappa,\theta}(u,v)=\int_{\Omega}\nabla u\nabla v dx+\int_{\po}\widetilde{u}\widetilde{v}d\kappa+\frac{1}{2}\int_{\partial\Omega\times\partial\Omega\setminus_d}(\widetilde{u}(x)-\widetilde{u}(y))(\widetilde{v}(x)-\widetilde{v}(y))d\theta
\end{equation*}
where $\widetilde{u}$ is a relative quasi-continuous representation
of $u$.

\end{theorem}



\begin{remark}
\begin{enumerate}
\item The proof of theorem \ref{th:1} is based on the techniques in the proofs of \cite[Theorem 3.1.1]{W02} and \cite[Theorem 4.1.2]{Br}.
\item One can consider a general $\theta$ not necessarly Radon, but
in this case the definition of admissible measures should change as
follows: The pair $(\kappa,\theta)$ is said admissible if $\kappa$ is admissible and for any separated compact sets $K_1$ and $K_2$ on $\po$ such that $\rcap(K_1)\rcap(K_2)=0$, we have $\theta(K_1\times K_2)=0$ (see \cite{AS} for more details).
\end{enumerate}
 
\end{remark}

\section{Markov property}

We denote by $-\Delta_{\kappa,\theta}$ the self-adjoint operator associated with the closure of the form $\mathcal E_{\kappa,\theta}$; i.e.
\begin{equation*}
  \left\{
  \begin{aligned}
D(\Delta_{\kappa,\theta}):=&\{u\in\mathfrak D\,:\,\exists v\in L^2(\Omega)\,:\, \mathcal E(u,\varphi)=(v,\varphi)\,\forall\varphi\in\mathfrak D\}\\
\Delta_{\kappa,\theta}:=&-v
\end{aligned}
 \right.
\end{equation*}If we choose $\varphi\in\mathcal D(\Omega)$, we obtain
\[
\langle-\Delta u,\varphi\rangle=\langle v,\varphi\rangle 
\]where $\langle,\rangle$ denotes the duality between $\mathcal D(\Omega)'$ and $\mathcal D(\Omega)$. Since $\varphi\in\mathcal D(\Omega)$ is arbitrary, it follows that 
\[
-\Delta u=v\quad \text{in } \mathcal D(\Omega)'
\]Thus $\Delta_{\kappa,\theta}$ is a realization of the Laplacian in $L^2(\Omega)$
\begin{proposition}
The operator $\Delta_{\kappa,\theta}$ generates a symmetric sub-Markovian semigroup on $L^2(\Omega)$; i.e., $e^{-\Delta_{\kappa,\theta}}\geq 0$ for all $t\geq 0$ and 
\[\|e^{-\Delta_{\kappa,\theta}}f\|_\infty\leq \|f\|_\infty\qquad (t\geq 0)
\]for all $f\in L^2(\Omega)\cap L^\infty(\Omega)$.
\end{proposition}

\begin{proof}
a) For $u\in D(\mathcal E)$. We have to show that $|u|\in D(\mathcal E)$ and $\mathcal E(|u|,|u|)\leq \mathcal E(u,u)$. Indeed, using the reverse triangle inequality $ \bigl\lvert|u|(x)-|u|(y)\bigr\rvert\leq|u(x)-u(y)|$ we get that

\begin{align*}
\mathcal E(|u|,|u|)&= \int_\Omega\left|\nabla |u|\right|\,dx+\int_{\po}|u|^2d\kappa+\int_{\po\times\po\setminus}\bigl\lvert|u|(x)-|u|(y)\bigr\rvert^2\theta(dx,dy)\\
& \leq\int_\Omega\left|\nabla u\right|\,dx+\int_{\po}|u|^2d\kappa+\int_{\po\times\po\setminus d}|u(x)-u(y)|^2\theta(dx,dy)\\
&=\mathcal E(u,u)
\end{align*}It follows from the first Beurling-Deny criterion that $e^{-t\Delta_{\kappa,\theta}}\geq 0$ for all $t\geq 0$ .

b) Let $0\leq u\in D(\mathcal E)$. We have
\begin{align*}
\int_{\po\times\po\setminus d}((u\wedge 1)(x)-(u\wedge 1)(y))^2\theta(dx,dy)& = \int_{\{u\leq 1\}\times \{u\leq 1\}\setminus d}(u(x)-u(y))^2\theta(dx,dy)\\&+\int_{\{u\leq 1\}\times \{u> 1\}\setminus d}(u(x)-1)^2\theta(dx,dy)\\&+\int_{\{u > 1\}\times \{u\leq 1\}\setminus d}(1-u(y))^2\theta(dx,dy)\\
\end{align*}On $\{u\leq 1\}\times \{u> 1\}\setminus d$ we have $0\leq u(x)\leq 1 < u(y)$. Thus $0\leq 1-u(x)< u(y)-u(x)$. Similarly, on $\{u> 1\}\times \{u\leq 1\}\setminus d$ we have $0\leq u(y)\leq 1 < u(x)$. Thus $0\leq 1-u(y)< u(x)-u(y)$. It follows that 
\begin{align*}
\int_{\po\times\po\setminus d}((u\wedge 1)(x)-(u\wedge 1)(y))^2\theta(dx,dy)& = \int_{\{u\leq 1\}\times \{u\leq 1\}\setminus d}(u(x)-u(y))^2\theta(dx,dy)\\&+\int_{\{u\leq 1\}\times \{u> 1\}\setminus d}(u(y)-u(x))^2\theta(dx,dy)\\&+\int_{\{u > 1\}\times \{u\leq 1\}\setminus d}(u(x)-u(y))^2\theta(dx,dy)\\
&=\int_{\po\times\po\setminus d}(u(x)-u(y))^2\theta(dx,dy)
\end{align*}
It follows that $u\wedge 1\in D(\mathcal E_{\kappa,\theta})$ and $\mathcal E_{\kappa,\theta}(u\wedge 1, u\wedge 1)\leq \mathcal E_{\kappa,\theta}(u,u)$. By the second Beurling-Deny criterion, the semigroup $e^{-t\Delta_{\kappa,\theta}}$ is $L^\infty-$contractive.
\end{proof}
The semigroup $(e^{-t\Delta_{\kappa,\theta}})_{t\geq 0}$ can be then extended to a contraction on $L^p(\Omega)$ for every $p\in[1,\infty[$. In addition we have the following result
\begin{corollary}
There exist consistent (i.e., $T_p(t)f=T_q(t)f$ if $f\in L^p(\Omega)\cap L^q(\Omega)$) $C_0-$semigroups $(e^{-t\Delta_{p,\kappa,\theta}})_{t\geq 0}$ on $L^p(\Omega)$, $1\leq p<\infty$, such that $\Delta_{2,\kappa,\theta}=\Delta_{\kappa,\theta}$.
\end{corollary}
\begin{proof}
This follows immediately from \cite[Theorem 1.4.1]{D89}
\end{proof}

\section{Domination}

Now, we want to verify if the semigroup generated by
$\Delta_{\kappa,\theta}$ is or is not sandwiched between the semigroups of Dirichlet Laplacian and the semigroup of Neumann Laplacian. The left hand side of the inequality is assured by the following theorem
\begin{theorem}
 We have \[e^{t\Delta_D}\leq e^{t\Delta_{\kappa,\theta}}\]
\end{theorem}
\begin{proof}
 By Ouhabaz's domination criterion, it suffices to prove that $H^1_0(\Omega)$ is an ideal of $D(\mathcal E_{\kappa,\theta})$ and $\mathcal E_{\kappa,\theta}(u,v)\leq\int_{\Omega}\nabla u\nabla v dx$ for all $u,v\in H^1_0(\Omega)_+$. We may assume that functions in $\Huntild$ are r.q.c.
\begin{enumerate}
\item  Let $u\in H^1_0(\Omega)$ and $v\in D(\mathcal E_{\kappa,\theta})$ such that
$0\leq v\leq u$. Since $\ombar$ is relatively open, it follows form
\cite[Lemma 2.1.4]{FOT94} that  $0\leq v\leq u$ r.q.c. on $\ombar$.
We have $u=0$ r.q.e. on $\po$, then $v=0$ r.q.e. on $\po$. Therefore
$v\in H^1_0(\Omega)$.

\item Let $u,v\in H^1_0(\Omega)_+$. Then $u=v=0$ r.q.e. on $\po$. Since
the pair $(\kappa,\theta)$ is admissible, it follows that $u=v=0$
$(\kappa+\hat{\theta})-$a.e. on $\Gamma_{\kappa,\theta}$. We finally obtain that
\begin{eqnarray*}
 \mathcal E_{\kappa,\theta}(u,v) &=&\int_{\Omega}\nabla u\nabla v dx+\int_{\po}uvd\kappa\\[0,2cm]& \qquad&+\int_{\po\times\po\setminus_d}(u(x)-u(y))(v(x)-v(y))d\theta \\[0,2cm]
   &=& \int_{\Omega}\nabla u\nabla v dx+\int_{\po}uvd(\kappa+2\hat{\theta})-2\int_{\po\times\po\setminus_d}u(x)v(y)d\theta\\[0,2cm]
   &=& \int_{\Omega}\nabla u\nabla v dx-2\int_{\po\times\po\setminus_d}u(x)v(y)d\theta\\[0,2cm]
   &\leq & \int_{\Omega}\nabla u\nabla v dx
\end{eqnarray*}
\end{enumerate}
and the proof is complete.
\end{proof}

The right hand side of the sandwiched inequality is not valid. We
have then the following theorem,

\begin{theorem}
We have \[e^{t\Delta_{\kappa,\theta}}\leq e^{t\Delta_N}\quad\text{ if
and only if } \quad\theta\equiv 0\]
\end{theorem}
\begin{proof} Similarly to the previous result, we use Ouhabaz's domination criterion and we may assume that functions in $\Huntild$ are r.q.c.
\begin{enumerate}
 \item Suppose that we have $e^{t\Delta_{\kappa,\theta}}\leq e^{t\Delta_N}$. By Ouhabaz's domination criterion, we know that $\int_{\Omega}\nabla u\nabla v dx\leq \mathcal E_{\kappa,\theta}(u,v)$ for
 all $u,v\in D(\mathcal E_{\kappa,\theta})_+$. Which means that for all $u,v\in D(\mathcal E_{\kappa,\theta})_+$ we have
\[
 \int_{\Omega}\nabla u\nabla v dx\leq \int_{\Omega}\nabla u\nabla v dx+
 \int_{\po}uvd(\kappa+\hat{\theta})-\int_{\po\times\po\setminus_d}u(x)v(y)d\theta
\]
Thus for all $u,v\in D(\mathcal E_{\kappa,\theta})_+$,
\[
 \int_{\po\times\po\setminus_d}u(x)v(y)d\theta\leq\int_{\po}uvd(\kappa+\hat{\theta})
\]
We choose  $u,v\in D(\mathcal E_{\kappa,\theta})_+$ such that
$\mathrm{supp}[u]\cap\mathrm{supp}[v]=\emptyset$, then we get
\[
 \int_{\po\times\po\setminus_d}u(x)v(y)d\theta\leq 0
\]
Which means that $\int_{\po\times\po\setminus_d}u(x)v(y)d\theta=0$
for all $u,v\in D(\mathcal E_{\kappa,\theta})_+$ such that
$\mathrm{supp}[u]\cap\mathrm{supp}[v]=\emptyset$. Thus $\theta\equiv
0$ because $\mathrm{supp}(\theta) \subset\po\times\po\setminus d$.
\item Now if $\theta\equiv 0$ the assertion follow from the sandwiched property proved in \cite{AW2} or \cite{A12}.
\end{enumerate}
\end{proof}
In the next theorem, we prove that nonlocal Robin semigroup can be dominated by a particular local semigroup.
\begin{theorem}\label{dom3}
Let $(\mu,\theta)$ be an admissible pair of measures. Then
\[
e^{-t\Delta_{{\kappa}+2\hat\theta}}\leq e^{-t\Delta_{\kappa,\theta}}\qquad (t\geq 0)
\]in the positive operators sense, where $\Delta_{{\kappa}+2\hat\theta}:=\Delta_{{\kappa}+2\hat\theta,0}$.
\end{theorem}

\begin{proof}
Similarly to the previous result, we use Ouhabaz's domination criterion and we may assume that functions in $\Huntild$ are r.q.c.
\begin{enumerate}
\item Let $u\in D(\mathcal E_{\kappa+2\hat\theta})$ and $v\in  D(\mathcal E_{\kappa,\theta})$ such that $0\leq v\leq u$ . We have
\begin{align*}
\mathcal E_{\kappa+2\hat\theta}(v)&=\int_{\Omega}|\nabla v|^2\,dx(u)+\int_{\po}|v|^2\,d(\kappa+2\hat\theta)\\
&=\int_{\Omega}|\nabla v|^2\,dx(u)+\int_{\po}|v|^2\,d(\kappa+2\hat\theta)-2\int_{\po\times\po}v(x)v(y)\theta(dx,dy)\\
&\qquad\qquad+2\int_{\po\times\po}v(x)v(y)\theta(dx,dy)\\
&= \int_{\Omega}|\nabla v|^2\,dx+\int_{\po}|v|^2\,d\kappa+\int_{\po\times\po}(v(x)-v(y))^2\theta(dx,dy)\\
&\qquad\qquad+2\int_{\po\times\po}v(x)v(y)\theta(dx,dy)\\
&=\mathcal E_{\kappa,\theta}(v)+2\int_{\po\times\po}v(x)v(y)\theta(dx,dy)\\
&\leq \mathcal E_{\kappa,\theta}(v)+\int_{\po\times\po}( u^2(x)+u^2(y))\theta(dx,dy)\\
&\leq \mathcal E_{\kappa,\theta}(v)+2\int_{\po} u^2(x)\hat\theta(dx)\\
&\leq \mathcal E_{\kappa,\theta}(v)+\mathcal E_{\kappa+2\hat\theta}(u)
\end{align*}This means that $v\in D(\mathcal E_{\kappa+2\hat\theta})$.
\item Let $u,\,v\in D(\mathcal E_{\kappa+2\hat\theta})_+ $. We have
\begin{align*}
\mathcal E_{\kappa,\theta}(u,v)&=\int_\Omega\nabla u\nabla v\,dx+\int_{\po}uv\,d\kappa+\int_{\po\times\po\setminus d}(u(x)-u(y)(v(x)-v(y))\,\theta(dx,dy)\\
&= \int_\Omega\nabla u\nabla v\,dx+\int_{\po}uv\,d\kappa+2\int_{\po}uv\,d\hat\theta-2\int_{\po\times\po\setminus d}u(x)v(y))\,\theta(dx,dy)\\
&=\mathcal E_{\kappa+2\hat\theta}(u,v)-2\int_{\po\times\po\setminus d}u(x)v(y))\,\theta(dx,dy)\\
&\leq \mathcal E_{\kappa+2\hat\theta}(u,v)
\end{align*}
\end{enumerate}
which completes the proof.
\end{proof}

\begin{corollary}
For any admissible measure $\nu$ on $\po$ such that $\nu\geq \kappa+2\hat\theta$ we have
\[
0\leq e^{-t\Delta_\nu}\leq e^{\Delta_{\kappa,\theta}}\quad (t\geq 0)
\]in the positive operators sense.

\end{corollary}
\begin{proof}
This follows immediately from Theorem \ref{dom3} and \cite[Theorem 3.2]{AW2}.
\end{proof}

\textbf{Open question:} for which admissible measure $\nu$, the semigroup $e^{-t\Delta_\nu}$ is the largest semigroup associated with local Robin boundary conditions and which is dominated by the nonlocal Robin semigroup. By Theorem 3.2 in \cite{AW2}, we know that $\nu\leq \kappa+2\hat\theta$. We conjecture that $\nu=\mu+2\hat\theta$ satisfies this property.

\section{Convergences}

Le $\Omega$ be a bounded open set in $\mathbb R^N$. Let $\kappa_n$ and $\theta_n$ be two sequences such that $(\kappa_n,\theta_n)$ is  an admissible pair of measures and $\kappa$ and $\theta$ two other meaures such that $(\kappa,\theta)$ is an admissible pair of measures. Consider the sub-Markovian $C_0-$semigroups $(e^{-t\Delta_{\kappa_n,\theta_n}})_{t\geq 0}$ and $(e^{-t\Delta_{\kappa,\theta}})_{t\geq 0}$ on $L^2(\Omega)$ generated respectively by $\Delta_{\kappa_n,\theta_n}$ and $\Delta_{\kappa,\theta}$. The question is, whether we have
\[
e^{-t\Delta_{\kappa_n,\theta_n}}\rightarrow e^{-t\Delta_{\kappa,\theta}}\quad \text{ as }\quad n\to\infty
\]if the measures $\kappa_n$ and $\theta_n$ converge to $\kappa$ and $\theta$, respectively, in an appropriate sense.  The results here generalize the results in \cite{W02} to the nonlocal case by using the same techniques.

\begin{proposition}
Let $\kappa_n$, $\theta_n$, $\kappa$ and $\theta$ be finite Borel measures such that $(\kappa_n,\theta_n)$ and $(\kappa,\theta)$ are admissible pairs of measures. Assume that $(\kappa_n,\theta_n)$ is monotone and $(\kappa_n,\theta_n)\to(\kappa,\theta)$ vaguely. Then $\Delta_{\kappa_n,\theta_n}\to \Delta_{\kappa,\theta}$ in the strong resolvent sense.

\end{proposition}

\begin{proof}
Without loss of generality we can assume that the sequence $(\kappa_n,\theta_n)$ is increasing. It follows that $\mathcal E_{\kappa_n,\theta_n}$ is positive increasing and the domains $D(\mathcal E_{\kappa_n,\theta_n})$ is decreasing. Let
\[
D_\infty:=\{u\in\bigcap_n D(\mathcal E_{\kappa_n,\theta_n})\,:\,\sup_n \mathcal E_{\kappa_n,\theta_n}(u,u)<\infty\}
\]and
\[
\mathcal E_{\infty}(u,u):=\lim_{n\to\infty} \mathcal E_{\kappa_n,\theta_n}(u,u)
\]
\begin{enumerate}
\item[~] \textbf{Claim 1:} $\mathcal E_\infty=\mathcal E_{\kappa,\theta}$.
\\As $H^1(\Omega)\cap C(\ombar)\subset D(\mathcal E_{\kappa_n,\theta_n})$ for all $n\geq 1$ we have that $H^1(\Omega)\cap C(\ombar)\subset D_\infty$. Let $u\in H^1(\Omega)\cap C(\ombar)$. Since $\kappa_n,\theta_n\to\kappa,\theta$ vaguely, it follows that $\lim_{n\to\infty}\mathcal E_{\kappa_n,\theta_n}(u,u)=\mathcal E_{\kappa,\theta}(u,u)$. Since $H^1(\Omega)\cap C(\ombar)$ is dense in $D_\infty$, we obtain that
\[
\lim_{n\to\infty}\mathcal E_{\kappa_n,\theta_n}(u,u):=\mathcal E_{\infty}(u,u)=\mathcal E_{\kappa,\theta}(u,u)\quad\forall u\in V_\infty.
\]Then $\mathcal E_{\kappa,\theta}(u,u)=\mathcal E_{\infty}(u,u)$ for all $u\in D_\infty$. 
\item[~] \textbf{Claim 2:} $D_\infty=D(\mathcal E_{\kappa,\theta})$.\\Since $\mathcal E_{\kappa,\theta}\leq \mathcal E_{\infty}$ we have that $D_\infty\subset D(\mathcal E_{\kappa,\theta})$. Let $u\in D(\mathcal E_{\kappa,\theta})$. It is clear that $u\in \cap_n D(\mathcal E_{\kappa_n,\theta_n})$. The density of $H^1(\Omega)\cap C(\ombar)$ in $D(\mathcal E_{\kappa,\theta})$ implies that $\sup_n \mathcal E_{\kappa_n,\theta_n}(u,u)=\mathcal E_{\kappa,\theta}(u,u)<\infty$ and thus $u\in D_\infty$ which implies that $D_\infty=D(\mathcal E_{\kappa,\theta})$.
\item[~] \textbf{Claim 3:} $(\Delta_{\kappa_n,\theta_n}+1)^{-1}\to (\Delta_{\kappa,\theta}+1)^{-1}$ weakly as $k\to\infty$.\\Since $\mathcal E_{\kappa_n,\theta_n}\leq \mathcal E_\infty$, by Proposition \cite[Theorem S.17 p.374]{RS80}, for $\varphi\in L^2(\Omega)$, we have
\[
(\varphi, (\Delta_{\kappa,\theta}+1)^{-1}\varphi)\leq (\varphi, (\Delta_{\kappa_n,\theta_n}+1)^{-1}\varphi).
\]Since $(\varphi, (\Delta_{\kappa_n,\theta_n}+1)^{-1}\varphi)$ is monotone decreasing, it follows that
\[
\lim_{k\to\infty} (\varphi, (\Delta_{\kappa_n,\theta_n}+1)^{-1}\varphi)=\inf_n (\varphi, (\Delta_{\kappa_n,\theta_n}+1)^{-1}\varphi)
\]has a nonzero value, so we can find a selfadjoint operator $C$ with zero kernel so that $(\varphi, (\Delta_{\kappa_n,\theta_n}+1)^{-1}\varphi)\to C$ weakly as $n\to\infty$ (see the proof of \cite[Theorem S.14 p.376]{RS80}). Let $b$ be the form associated with $B:=C^{-1}-1$. Since 
\[
(\Delta_{\kappa,\theta}+1)^{-1}\leq C\leq (\Delta_{\kappa_n,\theta_n}+1)^{-1}
\]By Proposition \cite[Theorem S.17 p.374]{RS80}, we have that $\mathcal E_{\kappa_n,\theta_n}\leq b\leq \mathcal E_{\kappa,\theta}$. Passing to the limit as $k\to\infty$, we obtain $\mathcal E_\infty=\mathcal E_{\kappa,\theta}\leq b\leq \mathcal E_{\kappa,\theta}$. Then $b=\mathcal E_{\kappa,\theta}$ and $B=\Delta_{\kappa,\theta}$. 
\item[~] \textbf{Claim 4:} $\Delta_{\kappa_n,\theta_n}\to \Delta_{\kappa,\theta}$ as $n\to\infty$ in the strong resolvent sense.\\By similar argument Claim 3 holds if $1$ is replaced by an arbitrary $\lambda>0$ and by analyticity we have weak convergence of the resolvent on $\mathbb C\setminus [0,\infty[$ and since weak resolvent convergence implies strong resolvent convergence (see \cite[Section VIII.7]{RS80}), this proves the claim.
\end{enumerate}

\end{proof}

\begin{example}
 Let $(\kappa,\theta)$ be an admissible pair of measures, we have $\Delta_{n\kappa,n\theta}\to\Delta_D$ as $n\to\infty$ in the strong resolvent sense.

\end{example}

\begin{definition}
Let $X$ be a metric space. Let $(F_n)_n$ be a sequence of functions from $X\to \overline{\mathbb R}:\mathbb R\cup\{-\infty,+\infty\}$ and $F:X\to\mathbb R$. We say that $(F_n)_n$ $\Gamma-$converges to $F$ in $X$ if the following conditions are satisfied.
\begin{enumerate}
\item For every $u\in X$ and for every sequence $u_n$ converging to $u$ in $X$
\[
F(u)\leq\liminf_{n\to\infty} F_n(u_n).
\]
\item For every $u\in X$ there exists a sequence $u_n$ converging to $u$ in $X$ such that 
\[
F(u)=\limsup_{n\to\infty} F_n(u_n).
\]
\end{enumerate}
\end{definition}

For each admissible pair of measures $(\kappa,\theta)$ we associate the following functional $F_{\kappa,\theta}$ defined on $L^2(\Omega)$ by letting

\[
F_{\kappa,\theta}:=
\left\{
\begin{array}{ll}
\overline{\mathcal E}_{\kappa,\theta}(u)\quad  &u\in  D(\overline{\mathcal E}_{\kappa,\theta})\\
\infty &u\in L^2(\Omega)\text{ but not in }D(\overline{\mathcal E}_{\kappa,\theta})
\end{array}
\right.
\]

\begin{definition}
Let $(\kappa_n,\theta_n)$ and $(\kappa,\theta)$ be pairs of admissible measures. We say that $(\kappa_n,\theta_n)$ $\gamma-$converges to $(\kappa,\theta)$ if the sequence of functionals $F_{\kappa_n,\theta_n}$ $\Gamma-$converges to the functional $F_{\kappa,\theta}$ in $L^2(\Omega)$.
\end{definition}

\begin{proposition}
Let $(\kappa_n,\theta_n)$ and $(\kappa,\theta)$ be pairs of admissible measures. Assume that $(\kappa_n,\theta_n)$ $\gamma-$converges to $(\kappa,\theta)$. Then $\Delta_{\kappa_n,\theta_n}$ converges in the strong resolvent sense to $\Delta_{\kappa,\theta}$.
\end{proposition}
\begin{proof}
Let $f\in L^2(\Omega)$, $u_n=\lambda R(\lambda,\Delta_{\kappa_n,\theta_n})f$ and $u=\lambda R(\lambda,\Delta_{\kappa,\theta})f$ where $\lambda>0$. We have to prove that $u_n\to u$ strongly in $L^2(\Omega)$ as $n\to\infty$. Remark that $u_n$ is a weal solution of the equation $-\Delta_{\kappa_n,\theta_n} u_n+\lambda u_n=\lambda f$. The Dirichlet principle (see \cite[Proposition IX.22]{B83}) says that $u_n$ is given by
\begin{align*}
\frac{1}{2}\min_{u\in \Huntild}\left(\mathcal E_{\kappa,\theta}(u)+\lambda\int_\Omega |u|^2\,dx-2\lambda\int_\Omega fu\,dx\right)
\end{align*}
By the definition of $\Gamma-$convergence of functionals there exists a sequence $v_n$ converging to $u$ in $L^2(\Omega)$ and $F_{\kappa,\theta}=\lim_{n\to\infty} F_{\kappa_n,\theta_n}(v_n)$.
This means that 
\[
F_{\kappa_n,\theta_n}(u_n)+\lambda\int_\Omega (u_n-f)^2\,dx \leq F_{\kappa_n,\theta_n}(v_n)+\lambda\int_\Omega (v_n-f)^2\,dx
\]Hence
\begin{align*}
\limsup_{n\to\infty}\left(F_{\kappa_n,\theta_n}(u_n)+\lambda\int_\Omega (u_n-f)^2\,dx\right)&\leq \lim_{n\to\infty}\left(F_{\kappa_n,\theta_n}(v_n)+\lambda\int_\Omega (v_n-f)^2\,dx\right)\\
&= F_{\kappa,\theta}(u)+\lambda\int_\Omega (u-f)^2\,dx
\end{align*}
It is clear that $u_n$ is a bounded sequence in $D(\mathcal E_{\kappa,\mu})$. Then there exists a subsequence which converges weakly in $D(\mathcal E_{\kappa,\mu})$ to a function $w\in D(\mathcal E_{\kappa,\mu})$. we may assume that $u_n\to w$ weakly in $D(\mathcal E_{\kappa,\mu})$. Then 
\begin{align*}
F_{\kappa,\theta}(w)+\lambda \int_\Omega (w-f)^2\,dx&\leq \liminf_{n\to\infty} \left(F_{\kappa_n,\theta_n}(u_n)+\lambda \int_\Omega (u_n-f)^2\,dx\right)\\
&\leq \limsup_{n\to\infty} \left(F_{\kappa_n,\theta_n}(u_n)+\lambda \int_\Omega (u_n-f)^2\,dx\right)\\
&\leq F_{\kappa,\theta}(u)+\lambda \int_\Omega (u-f)^2\,dx
\end{align*}
Since $u$ is unique, we have that $w=u$ and 
\[
F_{\kappa,\theta}(u)+\lambda \int_\Omega (u-f)^2\,dx=\lim_{n\to\infty}\left( F_{\kappa_n,\theta_n}(u_n)+\lambda \int_\Omega (u_n-f)^2\,dx  \right)
\]This implies that 
\[
\int_\Omega (u-f)^2\,dx=\lim_{n\to\infty} \int_\Omega (u_n-f)^2\,dx
\]and thus $u_n\to u$ strongly in $L^2(\Omega)$ as $n\to\infty$.

\end{proof}
\par\bigskip

\subsection*{Acknowledgments}
The idea to discuss the problem considered in this paper emerged during a research stay of the second author in Ulm, germany. The second author would like to thank all the members of the "Institüt für Angewandete Analysis" for their heart-warming welcome.

\end{document}